\documentclass[12pt,leqno,twoside]{amsart}
\usepackage[latin1]{inputenc}
\usepackage[T1]{fontenc}
\usepackage{amstext,amsmath,amscd,bezier,indentfirst,amsthm,amsgen,enumerate, geometry}
\usepackage[all,knot,arc,import,poly]{xy}
\usepackage{amsfonts}
\usepackage{amssymb}
\usepackage{latexsym}
\usepackage{epsfig}
\usepackage{graphicx}
\usepackage{srcltx}

\usepackage[colorlinks=true, pdfstartview=FitV, linkcolor=blue, citecolor=blue, urlcolor=blue]{hyperref} 
\theoremstyle{plain}
\newtheorem{theorem}{Theorem}[section]
\newtheorem{lemma}[theorem]{Lemma}
\newtheorem{proposition}[theorem]{Proposition}
\newtheorem{corollary}[theorem]{Corollary}

\theoremstyle{definition}
\newtheorem{definition}[theorem]{Definition}

\theoremstyle{remark}
\newtheorem{remark}[theorem]{\sc Remark}
\newtheorem{example}[theorem]{\sc Example}

\newcommand{\fin}{\hspace*{\fill}$\square$\vspace*{2mm}}

\def\bC{{\mathbb C}}
\def\bK{{\mathbb K}}
\def\bN{{\mathbb N}}
\def\bP{{\mathbb P}}

\def\bR{{\mathbb R}}
\def\bX{{\mathbb X}}

\def\cC{{\mathcal C}}

\def\cM{{\mathcal M}}

\def\cS{{\mathcal S}}

\def\cZ{{\mathcal Z}}
\def\ity{\infty}
\def\m{\setminus}

\def\ord{{\rm ord}}
\def\pr{\mathop{{\rm{pr}}}}
\def\reg{{\rm reg}}

\def\graph{{\rm graph}}
\def\Sing{{\rm Sing}}
\def\lcm{{\rm lcm}}

\def\const.{{\rm const.}}

\def\im{{\rm Im}}

\begin{document}

\title[Regularity at infinity of real mappings]{Regularity at infinity of real mappings and a Morse-Sard theorem}

\begin{abstract}
We prove a new Morse-Sard type theorem for the asymptotic critical values of semi-algebraic mappings and a new fibration theorem at infinity for $C^2$ mappings.   We  show the equivalence of three different types of regularity conditions which have been used in the literature in order to control the asymptotic behaviour of mappings. The central role of our picture is played by the $t$-regularity and its bridge toward the $\rho$-regularity which implies topological triviality at infinity.
 
\end{abstract}

\author{L.R.G. Dias}
\address{ICMC,
Universidade de S\~ao Paulo,  Av. Trabalhador S\~ao-Carlense, 400 -
CP Box 668, 13560-970 S\~ao Carlos, S\~ao Paulo,  Brazil  and Laboratoire Painlev\'e, UMR 8524 CNRS,
Universit\'e de Lille 1, \  59655 Villeneuve d'Ascq, France.}
\email{lrgdias@icmc.usp.br}

\author{M.A.S. Ruas}
\address{ICMC,
Universidade de S\~ao Paulo,  Av. Trabalhador S\~ao-Carlense, 400 -
CP Box 668, 13560-970 S\~ao Carlos, S\~ao Paulo,  Brazil.}
\email{maasruas@icmc.usp.br}

\author{M. Tib\u ar}
\address{Math\' ematiques, Laboratoire Painlev\'e, UMR 8524 CNRS,
Universit\'e Lille 1, \  59655 Villeneuve d'Ascq, France.}
\email{tibar@math.univ-lille1.fr}

\date{\today}
\thanks{LRG Dias acknowledges the Brazilian grants CAPES-Proc. 2929/10-04 and FAPESP-Proc. 2008/10563-4. MAS Ruas acknowledges support
from CNPq- Proc. 303774/2008-8 and FAPESP - Proc. 08/54222-6.
M Tib\u ar acknowledges support from the French grant ANR-08-JCJC-0118-01.
}
 


\maketitle
\pagestyle{plain}
\pagenumbering{arabic}

\section{Introduction}


 Let $f:\bK^{n}\rightarrow\bK^p$, for $n>p>0$ and $\bK = \bR$ or $\bC$, be a non-constant polynomial mapping. It is well-known that $f$ is a locally trivial topological fibration over the complement of the \textit{bifurcation set} $B(f)$, also called \textit{the set of atypical values}. The atypical values may come from the critical values but also from the asymptotic behaviour of the fibres. One can easily see this in the example $f(x,y)= x + x^2y$, where the value $0\in \bK$ is not critical but there is no trivial fibration in any neighbourhood of $0$. 
 
 A complete characterization of $B(f)$ is available only in the case $n=2$ and $p=1$, see \cite{HaLe} for $\bK = \bC$ and \cite{TZ} for $\bK = \bR$. One has therefore imagined various ways to approximate $B(f)$, essentially through the use of \textit{regularity conditions} at infinity. For $p=1$ and $\bK =\bC$, Broughton worked with a Palais-Smale type condition called \textit{tameness} \cite{Br}, later extended by N\' emethi \cite{Ne} and N\' emethi-Zaharia \cite{NZ} to quasi-tame and M-tame. Parusinski used the  \textit{Malgrange condition} \cite{Pa1}, which is a \L ojasiewicz type condition at infinity, and versions of it. Siersma and Tib\u ar worked with the \textit{$t$-regularity} (also called $t$-equisingularity) \cite{ST1, Ti-imrn}, which is a type of non-characteristic condition at infinity, see also \cite{Pa1}.
  Over $\bR$ and still for $p=1$, the $t$-regularity and the $\rho$-regularity were considered in \cite{Ti-cras, Ti-reg}. 
  One finds a detailed discussion of the relations between these conditions in \cite{NZ}, \cite{Du} for the complex setting and in \cite{Ti-reg, Ti-b} for the real and complex settings.

 Let us turn to the case of mappings, i.e. $p>1$. In his  study \cite{Ga-infinity} of polynomial mappings $f: \bC^n \to \bC^p$,  Gaffney defines a generalized Malgrange condition and proves that this yields a set $A_G$ of non-regular values containing $B(f)$. He uses the theory of integral closure of modules to relate his condition to a non-characteristic condition like  Parusinski's \cite{Pa1}.   
 
 In \cite{KOS}, Kurdyka, Orro and Simon considered $C^2$ semi-algebraic maps $f: \bK^n \to \bK^p$  and a metric type regularity condition which had been introduced by Rabier \cite{Ra}. They define a set of asymptotically critical values $K_\ity(f)$ and show the inclusion $A_{KOS} \supset B(f)$, where $A_{KOS} := f(\Sing f)\cup K_\ity(f)$. Remarkably, they prove that $K_\ity(f)$ is a semi-algebraic set of dimension $\le p-1$. 
 Their condition, which we shall call here \textit{$KOS$-regularity}, is a different extension of the Malgrange condition. It was established by Jelonek \cite{Je-man, Je} that these two generalizations of the Malgrange condition are actually equivalent, see Remark \ref{r:jelonek}.   
 
In each of the above cases, showing that some set $A$ of non-regular values is semi-algebraic (or algebraic, in the complex case) of dimension $\le p-1$ and contains $B(f)$, means to prove an \textit{asymptotic Morse-Sard type theorem} together with a \textit{fibration theorem} for the non-proper mapping $f$.

 In our paper, the central object is the \textit{$t$-regularity} in the setting of semi-algebraic $C^2$ mappings $f: \bR^n \to \bR^p$, which is a geometric grounded condition. 
Following Gaffney, the interpretation in terms of integral closure of modules  (Proposition \ref{p:intclosure_interpret}) allows one to prove that the $t$-regularity is  equivalent to the Malgrange-Gaffney condition reformulated over $\bR$ (Theorem \ref{t:t-reg=MalGaff}). Due to the technique, this holds within the class of ``fair'' maps \cite[p. 159 and Prop. 2]{Ga-infinity}, see also \S \ref{ss:fair}.

  Our Theorem \ref{t:main} tells that the $t$-regularity is equivalent to the $KOS$-regularity, providing a bridge between conditions of very different flavor.  This new result together with Jelonek's above mentioned theorem establishes a ``triangle'' of equivalences in Figure \ref{f:spec}. It also shows that the equivalence of the $t$-regularity with the Malgrange-Gaffney condition holds under general conditions, i.e. without the ``fair'' assumption and without using at all the integral closure interpretation, see Example \ref{ex:fair}. This is also a far reaching extension of the equivalence proved for $p=1$ in \cite{ST1, Pa2}.

We pursue by showing that $t$-regularity implies  \textit{$\rho_E$-regularity}, a Milnor type condition of transversality of $f$ to the Euclidean distance function $\rho_E$,  extending a result proved for $p=1$ in \cite{Ti-reg}. The $\rho$-regularity enables one to define the set of asymptotic non $\rho$-regular values $S(f)\subset \bR^p$, and the set $A_\rho := f(\Sing f) \cup S(f)$.

 Then Theorem \ref{t:asympSard}  tells that the subset $S(f)$  is closed semi-algebraic and of dimension $\le p-1$.  Moreover,  Theorem \ref{t:asympSard}(c) shows that there is a locally trivial fibration induced by $f$ on the complement $\bR^p \m A_{\rho_E}$, under the more general setting of a $C^2$-mapping $f : X \to \bR^p$ on a submanifold $X\subset \bR^n$.  Our fibration result is based on a new result, Proposition \ref{p:rho-reg_top-triv}, which is a fibration theorem ``at 
infinity'' , i.e. holding in the complement of a sufficiently large ball.
We show in Corollary \ref{c:kos_vs_sf} the inclusions $A_{\rho_E} \subset A_{KOS}$, and more particularly $S(f) \subset K_\ity(f)$, which inclusions may be strict, cf Example \ref{ex:paunescuzaharia}. Therefore Theorem \ref{t:asympSard} represents an asymptotic Morse-Sard type theorem which refines the one by Kurdyka, Orro and Simon \cite{KOS}. In particular, the key result $\dim K_\ity(f)\le p-1$  of \cite{KOS} is superseded by $\dim S(f)\le p-1$ (cf Theorem \ref{t:asympSard}(b)). Our proof is of a completely different flavor and is based  on the existence of \textit{partial Thom stratifications at infinity}, cf Definition \ref{d:partialThom}, and on the link between $t$-regularity and $\rho$-regularity.

Let us also point out that our definitions of regularity conditions as well as our statements involving them have a local-at-infinity counterpart, i.e. formulated at some fixed point $p_0$ on the boundary at infinity of the graph of $f$ in $\bP^n\times \bR^p$.

Figure \ref{f:spec} presents our results in a condensed manner.

\begin{figure}[hbtp]\label{f:spec}
\begin{center}
\epsfxsize=15cm
\leavevmode
\epsffile{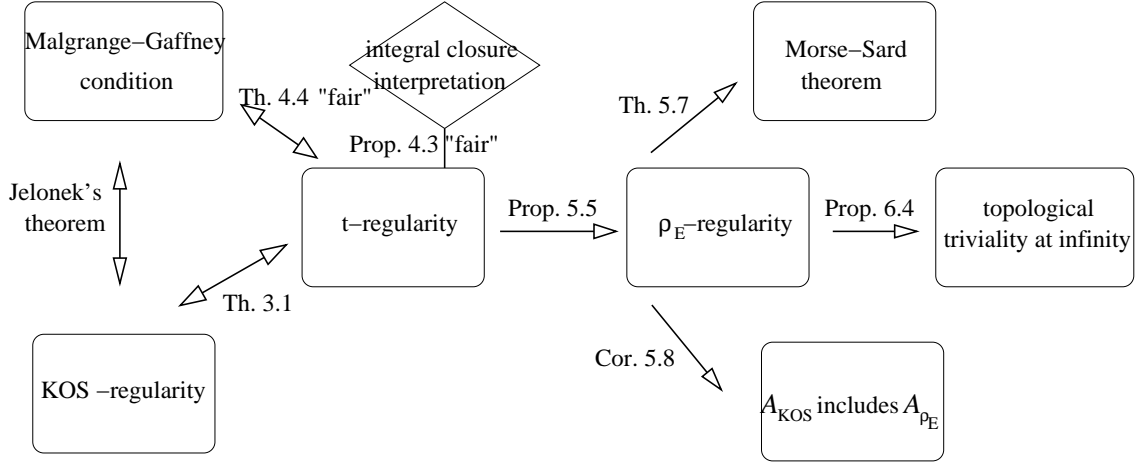}
\end{center}
\caption{{\em Synopsis}}  
\end{figure}



\bigskip

\section{Asymptotic regularity conditions for mappings}\label{s:t-reg}

We introduce the main definitions leading to the notion of $t$-regularity, then define the Malgrange-Gaffney regularity and the $KOS$-regularity.  
\subsection{Conormal spaces}\label{ss:conormalspaces}

Let $X\subset\bR^m$ be a real semi-algebraic subset. We denote by $X_{\reg}$ the set of regular points of $X$ and by $X_{\mathrm{sing}}$ the set of singular points of $X$. We assume that $X$ contains at least a regular point.

\begin{definition}\label{d:conorm}
Let 
\[C(X):=\mbox{closure}\{(x,H)\in X_{\reg}\times\check{\bP}^{m-1}\mid T_xX_{\reg}\subset H \}\subset \overline{X}\times\check{\bP}^{m-1}\]
be the \textit{conormal modification of $X$}.
Let $\pi:C(X)\to \overline{X}$ denote the projection. 
\end{definition}

\begin{definition}\label{d:relconorm}
  Let $g:X  \rightarrow \bR$ be an analytic function defined in some neighbourhood of $X$ in $\bR^m$. Let $X_0$ denote the subset of $X_\reg$ where $g$ is a submersion. The \emph{relative conormal space of} $g$ is defined as follows:
\[ C_{g}(X):=\mbox{closure}\{(x,H)\in X_0 \times\check{\bP}^{m-1}\mid T_x(g^{-1}(g(x)))\subset H\}\subset \overline{X}\times\check\bP^{m-1},\]
together with the projection $\pi: C_{g}(X) \to \overline{X}$, $\pi(x,H) = x$.
\end{definition}

For some $y\in  \overline{X}$ such that $g(y)=0$, let $C_{g,y}(X):=\pi^{-1}(y)$. The following result shows that $C_{g,y}(X)$ depends on the germ of $g$ at $y$ only up to multiplication by some invertible analytic function germ $\gamma$. It was stated for analytic $X$ but holds for semi-algebraic. 

\begin{lemma}\cite[Lemma 1.2.7]{Ti-b}\label{l:unit} \\
Let $\gamma:(\bR^m,y)\rightarrow \bR$ be an analytic function such that $\gamma(y)\neq 0$. Then $C_{\gamma g,y}(X)=C_{g,y}(X)$.
\fin
\end{lemma}


\subsection{Characteristic covectors and \texorpdfstring{$t$}{t}-regularity}\label{ss:char_covect}

Let $X\subset \bR^n$ be semi-algebraic and let $f:X \rightarrow\bR^p$ be a semi-algebraic $C^1$ mapping with $n>p$. We use coordinates  $(x_1,\ldots,x_n)$ for the space $\bR^n$ and coordinates $[x_0:x_1:\ldots:x_n]$ for the projective space $\bP^n$. We denote  by $H^{\infty}=\{[x_0:x_1:\ldots:x_n]\in\bP^n\mid x_0=0\}$ the hyperplane at infinity.

Let $\bX := \overline{\graph f}$ be the closure of the graph of $f$ in $\bP^n\times\bR^p$ and let $\bX^{\infty}:=\bX\cap (H^{\infty}\times\bR^p)$. One has the semi-algebraic isomorphism  $\graph f\simeq X$.

We consider the  affine charts $U_j\times\bR^p$ of $\bP^n\times\bR^p$, where $U_j=\{x_j\neq0\}$, $j=0,1,\ldots,n$. Identifying the chart $U_0$ with the affine space $\bR^n$, we have  $\bX\cap (U_0\times\bR^p)=\bX\setminus\bX^\infty=\graph f$ and $\bX^\infty$ is covered by the charts $U_1\times\bR^p,\ldots, U_n\times \bR^p$.

 If $g$ denotes the projection to the variable $x_0$ in some affine chart $U_j\times\bR^p$, then the
 relative conormal $C_{g}(\bX\backslash \bX^{\infty} \cap U_j\times\bR^p) \subset \bX \times \check\bP^{n+p -1}$ is well defined, with the projection $\pi(y,H) = y$.
Let us then consider the space $\pi^{-1}(\bX^\ity)$ which is well defined for every chart $U_j\times\bR^p$ as a subset of $C_{g}(\bX\backslash \bX^{\infty} \cap U_j\times\bR^p)$. By  Lemma \ref{l:unit}, the definitions coincide at the intersections of the charts. We therefore have: 

\begin{definition}\label{d:car_covectors_space}
We call  {\em space of characteristic covectors at infinity} the well-defined set $\mathcal{C}^{\ity} :=\pi^{-1}(\bX^\ity)$. For some $p_0\in \bX^\ity$, we denote $\cC^\ity_{p_0} := \pi^{-1}(p_0)$.
\end{definition}

Let $\tau:\bP^n\times\bR^p\rightarrow\bR^p$ denote the second projection. One defines the relative conormal space $C_{\tau}(\bP^n\times\bR^p)$ like in Definition \ref{d:relconorm} where  the function $g$ is replaced by the application $\tau$.

\begin{definition}\label{d:t-reg}
 We say that $f$ is \emph{$t$-regular} at $p_0\in\bX^\infty$ if $C_{\tau}(\bP^n\times\bR^p)\cap \cC^\ity_{p_0}=\emptyset$.
\end{definition}

%

\subsection{Interpretation of the \texorpdfstring{$t$}{t}-regularity in local charts}\label{ss:t-reg-interpret}

Let $f=(f_1,\ldots,f_p):\bR^n\to\bR^p$ be a semi-algebraic $C^1$ mapping and let  $p_0\in \bX^{\infty}$.  Up to  some linear change of coordinate one may assume that $p_0\in \bX^\infty\cap (U_n\times\bR^p)$.  In the intersection of charts  $(U_0\cap U_n)\times \bR^p$, one considers the change of coordinates  $x_1=y_1/y_0, \ldots,x_{n-1}= y_{n-1}/y_0$, $x_n = 1/y_0$, where $(x_1, \ldots , x_n)$ are the coordinates in $U_0$ and $(y_0, \ldots , y_{n-1})$ are those in $U_n$. Then for $i=1,\ldots,p$, we define: \[F_i(y,t)=F_i(y_0,y_1,\ldots,y_{n-1},t_1,\ldots,t_p):=f_i(\frac{y_1}{y_0},\ldots,\frac{y_{n-1}}{y_0},\frac{1}{y_0})-t_i\]
and $F(y,t) := (F_1(y,t),\ldots ,F_p(y,t))$. Then $\bX\cap((U_0\cap U_n)\times\bR^p)= \bigcap\limits_{i=1}^p \left\{ F_i(y,t)=0\right\}$.  

Let  $\vec{n_0}=\left(1,0,\ldots,0\right)\in\bR^n\times\bR^p$ denote a normal vector to the hypersurface $\{y_0 =\mbox{constant}\}$ and,  for $i=1,\ldots,p$,  let us consider a normal vector to $\{F_i=0\}$ at $(y,t)\in \bX\cap((U_0\cap U_n)\times\bR^p)$, as follows:
\[ \vec{n_i}(y,t)= \nabla F_i(y,t)= (\nabla_n F_i(y,t), \nabla_p F_i(y,t)),
\]
where
\[\nabla_n F_i(y,t):= \left(\frac{\partial F_i}{\partial y_0}(y,t),\ldots,\frac{\partial F_i}{\partial y_{n-1}}(y,t)\right),\ \  \nabla_p F_i(y,t):= \left(\frac{\partial F_i}{\partial t_1}(y,t),\ldots,\frac{\partial F_i}{\partial t_{p}}(y,t)\right).\]

By Definition \ref{d:t-reg}, $f$ is not $t$-regular at $p_0\in\bX^\infty$ if and only if there exists a sequence $\{(y_k,t_k)\}_{k\in \bN}\subset \bX\cap((U_0\cap U_n)\times\bR^p)$ such that $(y_k,t_k)\to p_0$ and  the tangent hyperplanes to the fibres of $g_{|\bX}$ at $(y_k,t_k)$ tend to a hyperplane $H$ such that its normal line has a direction  of the form $[0: \cdots : 0: b_1: \cdots: b_p]$ in $\bP^{n+p-1}$. More explicitly,
 there exists a sequence $\{(\psi_{0k},\psi_{1k},\ldots,\psi_{p_k})\}_{k\in \bN}\subset\bR^{p+1}$ such that the limit $\lim_{k\to\infty } \sum_{i=0}^p\psi_{ik}\vec{n_i}(y_k,t_k)$
of the linear combination of normal vectors $\vec{n_i}$
has the direction  $\vec{n}_H = [0 : 0:\ldots: 0:b_1:\ldots:b_p]\in \bP^{n+p-1}$.

\subsection{\texorpdfstring{$KOS$}{KOS}-regularity} \label{ss:def_KOS-reg}
Kurdyka, Orro and Simon defined in \cite{KOS} a regularity condition at infinity for $C^1$ mappings  $f:\bR^n\to\bR^p$ which contains the Rabier function $\nu$, cf \cite{Ra}, and compared it to the Kuo distance function \cite{Kuo}. Jelonek \cite{Je-man,Je} compared it to the generalized Gaffney function, see Remark \ref{r:jelonek}.

Let $f:\bR^n\to\bR^p$ be a nonconstant semi-algebraic $C^1$ mapping, where $n>p$. 

\begin{definition}\label{d:kos}\cite{KOS}
The set of \textit{asymptotic critical values of $f$} is defined as:
\begin{eqnarray*}
K_{\infty}(f)& := & \{t\in\bR^{p}\mid \exists \{ \mathrm{x}_{j}\}_{j\in \bN} \subset \bR^{n}, \lim_{j\to\infty}\|
\mathrm{x}_j\|=\infty, \\
 &  &  \lim_{j\to\infty}f(\mathrm{x}_j)= t \mathrm{\ and\ }\lim_{j\to\infty}\|\mathrm{x}_j\|\nu(\mathrm{D}f (\mathrm{x}_j))=0\}  
\end{eqnarray*}
where $\nu(A):=\inf_{\|\varphi\|=1}\|A^*(\varphi)\|$, for 
$A\in\mathcal{L}(\bR^n,\bR^p)$.
\end{definition}
We shall reformulate this condition in a localized version, at some point at infinity $p_0\in \bX^\ity$.

\begin{definition}\label{d:kos-reg}
Let  $\graph f \subset \bX \subset \bP^n\times \bR^p$ be a fixed embedding.
We say that  $p_0 \in \bX^\ity$ is an \textit{asymptotical critical point of $f$} if and only if there exists $\{\mathrm{x}_j\}_{j\in \bN}
\subset\bR^n \simeq \graph f$ such that $\lim_{j\to\infty} (\mathrm{x}_j,f(\mathrm{x}_j))= p_0$ and that $\tau(p_0)\in K_{\infty}(f)$.

We say that $t_0\in \bR^p$ is a \textit{$KOS$-regular value} of $f$ if $t_0\not\in K_{\infty}(f)$.
We say that $p_0\in \bX^\ity$ is a \textit{$KOS$-regular point} whenever $p_0$ is not an asymptotical critical point of $f$.
\end{definition}

Let us remark that the notion of $KOS$-regular point depends on the embedding of $\bR^n$, whereas the $KOS$-regular values  do not.


\subsection{Malgrange-Gaffney condition}\label{ss:malgrange-gaffney}
Following \cite{Ga-infinity},
let $f:\bR^n\to\bR^p$ be a semi-algebraic $C^1$ mapping and let $I$ be multi-index with $I=(i_1<i_2<\ldots<i_p)$.
One denotes by $M_I(f)$ the maximal minor of the Jacobian matrix of $f$ formed from the columns indexed by $I$ and by $M_J(j,f)$ the minor of the Jacobian matrix of size $(p-1)\times(p-1)$ using the columns indexed by $J$, and all the rows of the Jacobian matrix except for the $j$-th row. If $p=1$, then, by convention, $M_J(j,f)=1$.

\begin{definition}
Let $\{\mathrm{x}_k\}_{k\in\bN}$ be a sequence of points in $\bR^n$ and let us consider the following properties:

\noindent  $(L_1 )$  $\|\mathrm{x}_k\|\to \infty$ and $f(\mathrm{x}_k)\to t_0$, as $k\to\infty$.

\noindent $(L_2)$ $(\mathrm{x}_k, f(\mathrm{x}_k))\to p_0\in\bX^{\infty}$, as $k\to\infty$.

After \cite{Ga-infinity}, we say that the fibre $f^{-1}(t_0)$ verifies the  \textit{Malgrange-Gaffney condition} if there is $\delta>0$ such that, for any sequences of points $\{\mathrm{x}_k\}_{k\in\bN}$ with property $(L_1)$ one has
\begin{equation}\label{eq:malg.}
\|\mathrm{x}_k\|\frac{(\sum_I\|M_I(f)(\mathrm{x}_k)\|^2)^{\frac{1}{2}}}{(\sum_{J,j}\|M_J(j,f)(\mathrm{x}_k)\|^2)^{\frac{1}{2}}}\geq\delta.
\end{equation}

Localizing this condition, we shall say that $f$ verifies the Malgrange-Gaffney condition  at $p_0\in\bX^{\infty}$ if there is $\delta>0$ such that the above inequality holds  for any sequences of points with property $(L_2)$.
\end{definition}

It follows that $f^{-1}(t_0)$ verifies the Malgrange-Gaffney condition if and only if $f$ verifies the generalized Malgrange-Gaffney condition at any point $p_0\in\tau^{-1}(t_0)\cap\bX^{\infty}$.



\section{Equivalence of regularity conditions}\label{s:KOS}

 
 With the above definitions and notations, we have:
\begin{theorem}\label{t:main}
 Let $f:\bR^n\to\bR^p$ be a non-trivial semi-algebraic $C^1$ mapping, $n>p>0$. Let $p_0\in \bX^ \ity$. Then the following are equivalent:
\begin{enumerate}
\rm \item \it $f$ is $t$-regular at $p_0$.
\rm \item \it  $f$ is $KOS$-regular at $p_0$.
\rm \item \it $f$ verifies the Malgrange-Gaffney condition at $p_0$.

\end{enumerate}

\end{theorem}

\smallskip
\begin{proof}
 We shall prove (a)$\Leftrightarrow$(b).

One may assume (eventually after some linear change of coordinates) that $p_0\in \bX^\infty\cap (U_n\times \bR^p)$ and that we have  $|x_{n}|\geq |x_{i}|$, $i=1,\ldots,n-1$, for $\mathrm{x} \in \bR^n$ in some neighbourhood of $p_0$.

\smallskip
\noindent
\sloppy \textbf{(a)$\Rightarrow$(b)}. By definition, if  $p_0$ is an asymptotical critical point of $f$ and therefore $t_0 := \tau (p_0)$ is an asymptotical critical value,  then there exist sequences $\{\mathrm{x}_{k}:=\left(x_{1k},\ldots,x_{nk}\right) \}_{k\in \bN} \subset\bR^n$ and $\{\psi_k=(\psi_{1k},\ldots,\psi_{pk})\}_{k\in \bN}\subset \bR^p$ with $\|\psi_k\|=1$ and $\psi_k \to \psi$, such that $(\mathrm{x}_k,f(\mathrm{x}_k))\to p_0$ and that: 

\begin{eqnarray}\label{eq:kos1}
\|\mathrm{x}_k\| \left\| \left(\sum_{i=1}^p\psi_{ik}\frac{\partial f_i}{\partial x_1}(\mathrm{x}_k),\ldots, \sum_{i=1}^p\psi_{ik}\frac{\partial f_i}{\partial x_n}(\mathrm{x}_k) \right)\right\| \to 0.
\end{eqnarray}

Since for large enough $k$ we have $|x_{nk}|\geq |x_{ik}|$, $i=1,\ldots,n-1$,
one may replace in (\ref{eq:kos1}) $\|\mathrm{x}_k\|$ by $|x_{nk}|$, then multiply the sums by $x_{nk}$.

In the notations of \S \ref{ss:t-reg-interpret}, by changing coordinates within $U_0\cap U_n$, one has $y_0= 1/x_n,y_i=x_i/x_n$ and the relations:
\begin{equation}\label{eq:partials}
\left\{ \begin{array}{ll} 
  \frac{\partial F_j}{\partial y_i}(y,t)=x_n\frac{\partial f_j}{\partial x_i}(x), & \,\, 1\leq i \leq n-1,\,\,1\leq j\leq p,\\[1.5mm]
\frac{\partial F_j}{\partial t_l}(y,t)= -\delta_{l,j}, & \,\,\, 1\leq j, l \leq p,\\[1.5mm]
\frac{\partial F_j}{\partial y_0}(y,t)= - x_n(x_1\frac{\partial f_j}{\partial x_1}(x)+\ldots+x_n\frac{\partial f_j}{\partial x_n}(x)), & \,\,\, 1\leq j \leq p.
 \end{array} \right.
\end{equation}

 Then the condition (\ref{eq:kos1}) implies:

\begin{eqnarray}\label{eq:kos2}
\left\| \left(  \sum\limits_{i=1}^p\psi_{ik}\frac{\partial F_i}{\partial y_1}({\mathrm y}_k,{\mathrm t}_k),\ldots, \sum\limits_{i=1}^p\psi_{ik}\frac{\partial F_i}{\partial y_{n-1}}({\mathrm y}_k,{\mathrm t}_k) \right)\right\| \to 0.
\end{eqnarray}

The vector 
\[ \vec{n}_{H_k} :=  \left( 0,  \sum\limits_{i=1}^p\psi_{ik}\frac{\partial F_i}{\partial y_1}({\mathrm y}_k,{\mathrm t}_k),\ldots, \sum\limits_{i=1}^p\psi_{ik}\frac{\partial F_i}{\partial y_{n-1}}({\mathrm y}_k,{\mathrm t}_k), -\psi_{1k}, \ldots , -\psi_{pk} \right)\]
is a linear combination of the normal vectors $\vec{n}_i$ defined in \S \ref{ss:t-reg-interpret} with coefficients $\psi_{ik}$,
thus the hyperplanes $H_k$ are tangent to the levels of the function $g_{| \bX}$.
We have 
\[ \vec{n}_{H_k} \to \vec{n}= [ 0:0:\ldots :0:\psi_1:\ldots:\psi_p ] 
\]
which shows that the limit tangent hyperplane $H= \lim_{k\to\ity} H_k$, to which $\vec{n}$ is normal, belongs to  $\cC^\ity_{p_0}$. This implies that $f$ is not $t$-regular at $p_0$.
 
\smallskip
\noindent
\textbf{(b)$\Rightarrow$(a).}
Reciprocally, let $p_0\in \bX^\ity$ be not $t$-regular. Then there exist some sequence of points $\{(\mathrm{y}_k,\mathrm{t}_k)\}_{k\in \bN}\subset \bX\cap((U_0\cap U_n)\times \bR^p)$ tending to $p_0$, and a sequence of hyperplanes $H_k$ tangent to the levels of $g$ at $(\mathrm{y}_k,\mathrm{t}_k)$, such that $H_k \to H \in \cC^\ity_{p_0}$. This means that there exist sequences $\{\tilde \psi_k=(\tilde\psi_{1k},\ldots,\tilde\psi_{pk})\}_{k\in \bN}\subset \bR^p$ and  $\{\lambda_k\}_{k\in \bN}\subset \bR$ such that $\vec{n}_{H_k} = \lambda_k \vec{n}_0(\mathrm{y}_k,\mathrm{t}_k) + \sum_i \tilde\psi_{ik}\vec{n}_i(\mathrm{y}_k,\mathrm{t}_k)$ and 
that $\lim_{k\to\infty}\vec{n}_{H_k}=\left[0:0:\ldots :0:\tilde\psi_1:\ldots :\tilde\psi_p\right]$,
where $\left(\tilde\psi_1,\ldots ,\tilde\psi_p\right)\neq \left(0,\ldots,0 \right)$. 
By assumption we have that $\vec{n}_{H_k}$  is the vector:

 \begin{equation*}\label{eq:t-reg}
\left(\lambda_k+\sum\limits_{i=1}^p\tilde\psi_{ik}\frac{\partial F_i}{\partial y_0}(\mathrm{y}_k,\mathrm{t}_k),    \sum\limits_{i=1}^p\tilde\psi_{ik}\frac{\partial F_i}{\partial y_1}(\mathrm{y}_k,\mathrm{t}_k),\ldots, \sum\limits_{i=1}^p\tilde\psi_{ik}\frac{\partial F_i}{\partial y_{n-1}}(\mathrm{y}_k,\mathrm{t}_k),-\tilde\psi_{1k},\ldots,-\tilde\psi_{pk}\right).
\end{equation*}
We may actually take $\lambda_k := -\sum_{i=1}^p\tilde\psi_{ik}\frac{\partial F_i}{\partial y_0}(\mathrm{y}_k,\mathrm{t}_k)$ and we have, after dividing out by  $\mu_k := \| (\tilde\psi_{1k},\ldots,\tilde\psi_{pk})\|$, that 
$\lim_{k\to\infty}\vec{n}_{H_k}=\left(0,0,\ldots ,0, \psi_1,\ldots , \psi_p\right)$ where $ \psi_{ik} := \frac{\tilde\psi_{ik}}{\mu_k}$ and  $\|( \psi_{1k},\ldots , \psi_{pk})\| =1$. This implies that: 
\begin{equation}\label{eq:y_j}
 \lim_{k\to\infty} \sum_{i=1}^p \psi_{ik}\frac{\partial F_i}{\partial y_j}(\mathrm{y}_k,\mathrm{t}_k) = 0, 
\end{equation}
for any $1\le j\le n-1$.

 By using (\ref{eq:partials}), this is equivalent to: 
\begin{equation}\label{eq:limF_i}
\lim_{k\to\infty} x_{nk}\sum_{i=1}^p  \psi_{ik} \frac{\partial f_i}{\partial x_j}(\mathrm{x}_k) = 0
\end{equation} 
and  one has   $| x_{nk} | \ge \frac{1}{\sqrt{n}}\|\mathrm{x}_k\|$ for large enough $k$.
  Therefore, in order to get the limit (\ref{eq:kos1}) it remains to prove that (\ref{eq:limF_i}) is true for $j = n$. The rest of our argument is devoted to this proof.

From the relations (\ref{eq:partials}) we get $x_n\frac{\partial f_i}{\partial x_n}(x) =  - \sum_{j=0}^{n-1}y_j\frac{\partial F_i}{\partial y_j}(y,t)$
and therefore:
\[ 
 \sum_{i=1}^p  \psi_{ik} x_{nk}\frac{\partial f_i}{\partial x_n}(\mathrm{x}_k) = - \sum\limits_{j=1}^{n-1}  \sum_{i=1}^p y_{jk} \psi_{ik}\frac{\partial F_i}{\partial y_j} (\mathrm{y}_k,\mathrm{t}_k)
- \sum_{i=1}^p   \psi_{ik}  y_{0k}\frac{\partial F_i}{\partial y_0 }(\mathrm{y}_k,\mathrm{t}_k).
\]

We show that both terms of the right hand side tend to zero.
First we have: 
  \[ \left| \sum_{j=1}^{n-1}  \sum_{i=1}^p y_{jk}  \psi_{ik}\frac{\partial F_i}{\partial y_j} (\mathrm{y}_k,\mathrm{t}_k) \right| \le \left\| \frac{\mathrm{x}_k}{x_{nk}}\right\| \left\| (\sum_{i=1}^p \psi_{ik}\frac{\partial F_i}{\partial y_1}(\mathrm{y}_k,\mathrm{t}_k),\ldots, \sum_{i=1}^p \psi_{ik}\frac{\partial F_i}{\partial y_{n-1}}(\mathrm{y}_k,\mathrm{t}_k))\right\| .\]
Since by hypothesis we have $|y_{jk}| = |\frac{x_{jk}}{x_{nk}}| \le 1$ for large enough $k$, we get from (\ref{eq:limF_i}) that the right hand side tends to 0 as $k\to \ity$.

Let us assume that the following inequality holds for large enough $k\gg 1$, the proof of which will be given below:
\begin{equation}\label{eq:gpeq}
\left\|\sum_{i=1}^p   \psi_{ik} y_{0k} \frac{\partial F_i}{\partial y_0}\right\| \ll \left\| ( \sum_{i=1}^p   \psi_{ik} \frac{\partial F_i}{\partial y_1},\ldots ,\sum_{i=1}^p   \psi_{ik} \frac{\partial F_i}{\partial y_{n-1}},\sum_{i=1}^p   \psi_{ik}\frac{\partial F_i}{\partial t_1},\ldots,\sum_{i=1}^p   \psi_{ik} \frac{\partial F_i}{\partial t_p} )\right\|.
\end{equation}

Then, by using (\ref{eq:y_j}), (\ref{eq:gpeq})  and the equality $\sum_{i=1}^p   \psi_{ik} \frac{\partial F_i}{\partial t_j} = -  \psi_{jk} $ for any $1\le j\le p$ (implied by (\ref{eq:partials})), we get:
\[
\left\| \sum_{i=1}^p  \psi_{ik} y_{0k} \frac{\partial F_i}{\partial y_0} \right\| \ll  \| \psi_{k} \|=1,
\]
which shows that $\lim_{k\to\infty}  \sum_{i=1}^p  \|  \psi_{ik} y_{0k} \frac{\partial F_i}{\partial y_0}  (\mathrm{y}_k,\mathrm{t}_k) \| = 0$. This completes our proof of the relation  (\ref{eq:kos1}) showing that $p_0$ is not $KOS$-regular.

\smallskip

Let us now give the proof of (\ref{eq:gpeq}). If this were not true,  there exists $\delta >0$ such that for $k\gg 1$ one has:
\begin{equation}\label{eq:eqy0} \frac{\left\|\sum_{i=1}^p   \psi_{ik} y_{0k} \frac{\partial F_i}{\partial y_0} (\mathrm{y}_k,\mathrm{t}_k)\right\|} {\left\| (\sum_{i=1}^p   \psi_{ik} \frac{\partial F_i}{\partial y_1},\ldots ,\sum_{i=1}^p   \psi_{ik} \frac{\partial F_i}{\partial y_{n-1}},\sum_{i=1}^p    \psi_{ik}\frac{\partial F_i}{\partial t_1},\ldots,\sum_{i=1}^p   \psi_{ik} \frac{\partial F_i}{\partial t_p} ) (\mathrm{y}_k,\mathrm{t}_k)\right\|}>\delta .
\end{equation} 
Then the set $\mathcal{W}=\{ ((y,t),\psi)\in ((U_n \cap U_0)\times\bR^p\times\bR^p) \cap (\bX \times S_1^{p-1})\mid \mbox{(\ref{eq:eqy0}) holds for } ((y,t),\psi) \}$ is a semi-algebraic set. We have $((\mathrm{y}_{k},\mathrm{t}_k), \psi_{k})\in \mathcal{W}$ for $k\gg 1$, thus  $(p_0,\psi)\in\overline{\mathcal{W}}$. Then, by the curve selection lemma \cite[p. 25]{Mi} there exists an analytic curve $\lambda=(\phi,\psi):[0, \varepsilon [ \to \overline{\mathcal{W}}$ such that $\lambda(]0, \varepsilon[ ) \subset\mathcal{W}$ and $\lambda(0) =(p_0,\psi$). We denote $\phi(s)=(y_0(s),y_1(s),\ldots,y_{n-1}(s),t_1(s),\ldots,t_p(s))$ and $\psi(s)=(\psi_1(s),\ldots,\psi_p(s))$. Since $F(\phi(s))\equiv 0$, we have: 
\begin{equation*}
0=\frac{d}{ds}F(\phi(s))=y'_0(s)\frac{\partial F}{\partial y_0}(\phi(s))+\sum_{i=1}^{n-1} y'_i(s)\frac{\partial F}{\partial y_i}(\phi(s))+\sum_{i=1}^{p}t'_i(s)\frac{\partial F}{\partial t_i}(\phi(s)). 
\end{equation*}

Multiplying by $\psi(s)$ we obtain:
\begin{equation*}\label{eq:Fpsi}
-y'_0(s)\sum\limits_{i=1}^p\psi_i(s)\frac{\partial F_i}{\partial y_0}(\phi(s))=\sum_{j=1}^{n-1}y'_j(s)\sum_{i=1}^p\psi_i(s)\frac{\partial F_i}{\partial y_j}(\phi(s)) +\sum_{j=1}^pt'_j(s)\sum_{i=1}^p\psi_i(s)\frac{\partial F_i}{\partial t_j}(\phi(s)).
\end{equation*}
Since $\phi$ is analytic, thus bounded at $s=0$, by applying the Cauchy-Schwarz inequality one finds a constant $C >0$ such that:
\begin{multline}\label{eq:eqy02}
\left|  (y'_0(s) \sum_{i=1}^p\psi_i(s) \frac{\partial F_i}{\partial y_0}(\phi(s)) \right| \le  \\ C  \left\| (\sum_{i=1}^p  \psi_{i}    \frac{\partial F_i}{\partial y_1}(\phi  ),\ldots ,\sum_{i=1}^p  \psi_{i}   \frac{\partial F_i}{\partial y_{n-1}}(\phi  ),\sum_{i=1}^p  \psi_{i}  \frac{\partial F_i}{\partial t_1}(\phi   ),\ldots,\sum_{i=1}^p  \psi_{i}  \frac{\partial F_i}{\partial t_p} (\phi  ))(s)\right\|.
\end{multline}
 We have $l :=\mathrm{ord}_s y'_0(s) \ge 0$ and $\ord_s y_0(s)=l+1 \ge 1$ since $y_0(0)=0$, thus   $\left|y_0(s)\sum_{i=1}^p\psi_i(s)\frac{\partial F_i}{\partial y_0}(\phi(s))\right| \ll \left|y'_0(s)\sum_{i=1}^p\psi_i(s)\frac{\partial F_i}{\partial y_0}(\phi(s))\right|$, which, together with (\ref{eq:eqy02}), gives:
\begin{multline*}
\left\|\sum_{i=1}^p  \psi_{i}(s) y_0(s) \frac{\partial F_i}{\partial y_0}(\phi(s))\right\|\ll\\ \left\| (\sum_{i=1}^p  \psi_{i} \frac{\partial F_i}{\partial y_1}(\phi),\ldots ,\sum_{i=1}^p  \psi_{i} \frac{\partial F_i}{\partial y_{n-1}}(\phi),\sum_{i=1}^p  \psi_{i} \frac{\partial F_i}{\partial t_1}(\phi),\ldots,\sum_{i=1}^p  \psi_{i}  \frac{\partial F_i}{\partial t_p}(\phi) )(s)\right\|.
\end{multline*}

This contradicts our assumption that $(\phi(s),\psi(s))\in\mathcal{W}$, for $s\in\, ]0,\varepsilon[$.

\smallskip 
\noindent
The equivalence (b)$\Leftrightarrow$(c) has been proved by Jelonek \cite{Je-man}, \cite{Je} in a slightly more general setting.

\end{proof}


\section{Integral closure interpretation of \texorpdfstring{$t$}{t}-regularity and Malgrange-Gaffney Condition}\label{s:malgrange}\label{ss:intcl_interpret_t-reg}


In the setting of complex polynomial mappings $f:\bC^n\to\bC^p$, one observed that $t$-regularity has an integral closure interpretation,  cf \cite[Remark 2.9]{Ti-imrn} for $p=1$ and  \cite{Ga-infinity} for $p\ge 1$. Actually Gaffney  described in terms of the integral closure of modules a non-characteristic condition for $p\ge 2$ \cite[Def.1]{Ga-infinity} which turns out to be equivalent to the $t$-regularity. His technique works under the supplementary assumption that $f$ is ``fair'', see below. 
Let us first give the  interpretation of the $t$-regularity of polynomial mappings $f:\bR^n\to\bR^p$ which are fair. We then prove the real counterpart of Gaffney's result \cite[Theorem 17]{Ga-infinity}.

\subsection{Real integral closure of modules}\label{ss:i.c.}
Let us denote by $\mathcal{A}_n$ the local ring of real analytic function germs at the origin in $\bR^n$ and by $\mathcal{A}^p_n$ the free $\mathcal{A}_n$-module of rank $p$. One  denotes by $\mathcal{A}_{X,x}$ the local ring of real analytic function germs on the real analytic space germ $(X,x)$ and by $\mathcal{A}_{X,x}^p$ the free $\mathcal{A}_{X,x}$-module of rank $p$.

\begin{definition}\label{d:i.c.}\cite[Def. 4.1]{Gaintclosure} Let $(X,x)\subset (\bR^n, x) $ be a real analytic germ and let $M$ be a $A_{X,x}$-submodule of $\mathcal{A}^p_{X,x}$. The {\em real integral closure of} $M$, denoted by $\overline{M}$, is the set of elements $h\in\mathcal{A}^p_{X,x}$ such that for any analytic path  $\phi:(\bR,0)\to(X,x)$, we have $h\circ\phi\in \mathcal{A}_1(\phi^*(M))$, where $\mathcal{A}_1(\phi^*(M))$ denotes the $\mathcal{A}_1$-submodule of $\mathcal{A}^p_1$ generated by the elements $w\circ\phi$, $\forall w\in M$.
\end{definition}

In the complex setting one has some equivalent definitions, see for example \cite[Prop. 1.7, 1.11]{Gaintclosure}, which hold in the real setting except of \cite[Prop. 1.7]{Gaintclosure} where it is necessary to assume that the regular points of $X$ are dense in $X$, see \cite[p. 318]{Gaintclosure}.

There is an interpretation of the conormal in terms of the integral closure of modules in the complex setting, see for instance \cite{Gaaureoles,GTW}. The next result was formulated in \cite[Prop. 4]{Ga-infinity} in the complex setting but one can show that it holds over the reals if one assumes that the regular points of $X$ are dense, see \cite{R-thesis} for the details and generalization. Under the notations of \S \ref{ss:conormalspaces} one has:

\begin{lemma}\label{p:hyperplane_limits}
Let $(X,x)\subset(\bR^m,x)$ be an equidimensional real analytic germ defined by a mapping germ $\tilde F : (\bR^m,x)\to (\bR^p,0)$. Suppose that the regular points of $X$ are dense in $X$. Let $g:(\bR^m,x)\to \bR$ and denote  $G := (\tilde F,g)$. Let $V\subset \bR^m$ be a linear subspace. Then the following are equivalent:
\begin{enumerate}
\rm \item \it There exist $H\in C_{g,x}(X)$ such that $H\supset V$.
\rm \item \it $\overline{JM_X(G)_V}\subsetneq \overline{ JM_X(G)}$, where $JM_X(G)_V$ denotes the $\mathcal{A}_{X,x}$-submodule  of $\mathcal{A}_{X,x}^{p+1}$ generated by $\left\{v_1\frac{\partial G}{\partial z_1}+ \ldots+ v_m \frac{\partial G}{\partial z_m};\forall v=(v_1,\ldots,v_m)\in V \right \}$, and $JM_X(G)$ denotes the $\mathcal{A}_{X,x}$-submodule  of $\mathcal{A}_{X,x}^{p+1}$ generated by $\left\{ \frac{\partial G}{\partial z_1},\ldots,\frac{\partial G}{\partial z_m}\right\}$.
\end{enumerate}
\fin
\end{lemma}

\subsection{The ``fair'' condition}\label{ss:fair}

Let $f =(f_1, \ldots , f_p) :\bR^n\to\bR^p$ be a polynomial mapping  and let $\tilde{f_i}(x_0,x_1,\ldots,x_n)$ denote the homogenization of $f_i$ of degree $d_i := \deg f_i$, for $ 1\leqslant i\leqslant p$.
Let:
\[ \cZ =\bigcap_{i=1}^p \left\{ \tilde F_i(x_0,x_1,\ldots,x_{n},t_1,\ldots,t_p) :=\tilde{f}_i(x_0,x_1,\ldots,x_{n})-t_ix_{0}^{d_i}=0\right\} \subset \bP^n \times \bR^p\] 
and compare to the definition of $\bX$ in \S \ref{s:t-reg} and to the notations in \S \ref{ss:t-reg-interpret}.

One notices the inclusion $\bX\subset \cZ$ and the equality $\bX \m H^\ity \times \bR^p = \cZ \m H^\ity \times \bR^p$. 
At any $p_0 \in \bX^\ity$ we may use  the equations of $\cZ$ in the open subspace $\cZ \m H^\ity \times \bR^p \simeq \bR^n$.
 However, one does not have the equality $\bX = \cZ$ in general. Gaffney works in \cite{Ga-infinity} with the space $\cZ$ assuming the equality $\cZ=\bX$ which he translates by ``\textit{$f$ is fair}'' \cite[p. 158]{Ga-infinity}. This is imposed by the theory of integral closure of modules since ``fair'' implies that the regular points of $\cZ$ are dense in $\cZ$ and that $\cZ$ is equidimensional analytic. See  Example \ref{ex:fair} in which $\bX\not= \cZ$ and therefore
this technique does not apply, whereas the statement and proof of the equivalence result Theorem \ref{t:main} hold true.

\subsection{\texorpdfstring{$t$}{t}-regularity and fair polynomial mappings}\label{ss:i.c.xfair}

Let $f =(f_1, \ldots , f_p) :\bR^n\to\bR^p$ be a fair polynomial mapping. As in \S  \ref{ss:char_covect}, we consider the charts $U_j\times\bR^p$ of $\bP^n\times\bR^p$, where $U_j=\{x_j\neq0\}$, $j=0,1,\ldots,n$ and we identify the chart $U_0$ with the affine space $\bR^n$. 

Let $p_0\in\bX^\infty$. Up to  some linear change of coordinate one may assume that $p_0\in \bX^\infty\cap (U_n\times\bR^p)$. In the chart  $U_n\times \bR^p$, one considers the change of coordinates  $y_0=x_0/x_n, \ldots,y_{n-1}= x_{n-1}/x_n$. In this coordinate system and since $f$ is fair, one has \[\bX\cap (U_n\times\bR^p)=\bigcap_{i=1}^p \left\{\tilde F_i(y,t)=\tilde f_i(y_0,y_1,\ldots,y_{n-1},1)-t_iy_0^{d_i}=0\right\}.\]

\begin{proposition}\label{p:intclosure_interpret}
Let $f:\bR^n\to\bR^p$ be a fair polynomial map. Then $f$ is $t$-regular at $p_0\in \bX^\ity$ if and only if
one of the following equivalent conditions is satisfied:
\begin{eqnarray}\label{eq:integralcl}
\partial \tilde F/\partial t_i\in\overline{\{\partial \tilde F/\partial y_1,\ldots,\partial \tilde F/\partial y_{n-1}\}}, \ \ \forall i=1,\ldots,p.
\end{eqnarray}
\begin{equation}\label{eq:integralcl+}
\partial \tilde F/\partial t_i\in\overline{\{y_0\partial \tilde F/\partial y_0, \partial \tilde F/\partial y_1,\ldots,\partial \tilde F/\partial y_{n-1}\}}, \ \ \forall i=1,\ldots,p.
\end{equation}
\end{proposition}
\begin{proof}
Let $V :=\bR^n\times 0\subset\bR^n\times\bR^p$. By Definitions \ref{d:car_covectors_space} and \ref{d:t-reg}, $p_0$ is a $t$-regular point if and only if  there are no hyperplanes $H \supset V$ such that $H \in \cC^\ity_{p_0}$.  By Lemma \ref{p:hyperplane_limits} applied for $G= (\tilde F, g)$ and $V$, this is equivalent to the following: 

\begin{eqnarray}\label{eq:closureG}
\frac{\partial G}{\partial t_i}\in \overline{\left\{\frac{\partial G}{\partial y_0},\frac{\partial G}{\partial y_1},\ldots,\frac{\partial G}{\partial y_{n-1}}\right\}}, \mbox{ for $i=1,\ldots,p$}.
\end{eqnarray}

By Definition \ref{d:i.c.}, (\ref{eq:closureG}) means that for any $\phi:(\bR,0)\to (\bX,p_0)$, there exist  $\lambda_0, \lambda_1,\ldots,$ $\lambda_{n-1}\in \mathcal{A}_1$ such that:

\[ \left( \begin{array}{c}
  \frac{\partial \tilde F}{\partial t_i} (\phi(s))\\ 
0 
   \end{array}
 \right)=\lambda_0(s)\left(
   \begin{array}{c}
      \frac{\partial\tilde F}{\partial y_0}(\phi(s)) \\ 
1 
  \end{array}
     \right)+ \lambda_1 (s)\left(
    \begin{array}{c}
     \frac{\partial\tilde F}{\partial y_1}(\phi(s)) \\ 
0 
      \end{array}
  \right)+ \cdots + \lambda_{n-1}(s) \left(
     \begin{array}{c}
      \frac{\partial\tilde F}{\partial y_{n-1}}(\phi(s)) \\ 
0 
       \end{array} \right)
 \]
\noindent
which is in turn equivalent to (\ref{eq:integralcl}). 

That (\ref{eq:integralcl}) implies (\ref{eq:integralcl+}) is obvious. The converse is the same as Gaffney's proof of \cite{Ga-infinity}, holds over the reals too and is based on Parusinski's proof \cite[Lemma 3.2]{Pa1}.\end{proof}


\begin{theorem}\label{t:t-reg=MalGaff}
 A fair polynomial mapping $f:\bR^n\rightarrow\bR^p$ satisfies the Malgrange-Gaffney condition at
$p_0\in\bX^{\infty}$ if and only if $f$
is $t$-regular at this point.
\end{theorem}
\begin{proof}
That $f$ is $t$-regular at $p_0$ is equivalent to (\ref{eq:integralcl+}) of Proposition \ref{p:intclosure_interpret}. In turn, by  \cite[Prop. (1.7)]{Gaintclosure} which hold over $\bR$ too since $f$ is fair, see \cite[p. 318]{Gaintclosure} and \cite{R-thesis}, this is equivalent to:
\begin{equation}\label{eq:i.c.xmalg}
 y^{d_j}_0M_J(j,\tilde F)\in\overline{\langle M_I(\tilde F)\rangle}, j=1,\ldots,p,
\end{equation}
where $ \langle M_I(\tilde F)\rangle$ denotes the ideal generated by the $p\times p$ minors of the matrix whose columns are $(y_0\partial \tilde F/\partial y_0, \partial \tilde F/\partial y_1,\ldots,\partial \tilde F/\partial y_{n-1})$, and $M_J(j,\tilde F)$ is a maximal minor of the same matrix  with the $j$th row deleted.

Using \cite[Prop 4.2]{Gaintclosure}, one has that (\ref{eq:i.c.xmalg}) is equivalent to the existence of $C >0$ such that:
\[
 \sup_{J,j}\|y_0^{d_j}\|\|M_J(j,\tilde F)(y,t)\|\leq C\sup_I\|M_I(\tilde F)(y,t)\|. 
\]
Dividing both sides by $\|y_0^k\|$, where $k=\sum\limits_{l=1}^p(d_l-1)$, and using properties of determinant and the following relations between the partials of $f$ and the partials of $\tilde F$:
\begin{equation*}\label{eq:part.hom}
\left\{ \begin{array}{ll} 
  \frac{\partial\tilde F_j}{\partial y_i}/ y_{0}^{d_j-1}=\frac{\partial f_j}{\partial x_i}, & \,\, 1\leq i \leq n-1,\,\,1\leq j\leq p,\\[1.5mm] \frac{\partial\tilde F_j}{\partial t_l}= - y_0^{d_j}\delta_{l,j}, & \,\,\, 1\leq j, l \leq p,\\[1.5mm]
\frac{\partial\tilde F_j}{\partial y_0}/y_0^{d_j-1}= - (x_1\frac{\partial f_j}{\partial x_1}+\ldots+x_n\frac{\partial f_j}{\partial x_n}), & \,\,\, 1\leq j \leq p,
 \end{array} \right.
\end{equation*}
we obtain
\begin{equation}\label{eq:malg.1}\sup_{J,j}\|1/x_n\|\|M'_J(j,f)(x)\|\leq C\sup_I\|M'_I(f)(x)\|,
\end{equation}
where the $M'$ matrices are defined as follows. If $I=(i_1<\ldots<i_p)$ with $i_1\neq 1$, i.e., if the matrix indexed by $I$ does not contain the column  $y_0\partial \tilde F/\partial y_0$ then $M'_I(f)=M_I(f)$, and otherwise, one replaces the column vector with the $\partial f/\partial x_n$ terms by $\sum\limits_{l=1}^n \frac{x_l}{x_n}\partial f/\partial x_l$. A similar substitution should be made to define the $M'_J(j,f)$ terms.

Using the fact that $\|x_n\|\geq \|x_j\|,$ for $j=1,\ldots,n$, the inequality (\ref{eq:malg.1}) is equivalent to generalized Malgrange-Gaffney condition (\ref{eq:malg.}).
\end{proof}

\begin{remark}\label{r:jelonek}
 Jelonek \cite{Je, Je-man} proved directly that the  Malgrange-Gaffney condition (actually in a more general form) is equivalent to the $KOS$-regularity, by comparing the corresponding functions to the Kuo distance function \cite{Kuo}.
\end{remark}



\section{Atypical values and fibrations at infinity}\label{s:fib}

\begin{definition}\label{d:typical}
Let $X\subset \bR^n$ be a  submanifold and let $f:X\to\bR^p$ be a $C^1$-mapping, where $p< \dim X$.
One says that $t_0\in\bR^p$ is a \textit{typical value} of $f$ if there exists a disk $D\subset \bR^p$ centered at $t_0$ such that the restriction $f_| : f^{-1}(D) \to D$ is a topologically trivial fibration; otherwise one says that $t_0$ is \textit{atypical}. We denote by $B(f)$ the \textit{set of atypical values} of $f$ (or the \textit{bifurcation locus} of $f$).
\end{definition}
Note that by definition $B(f) \supset \overline{\im f}\m \mathring{\im f}$ and that 
 the fibres of $f$ depend on the  connected components of $\bR^p \m B(f)$.

\subsection{\texorpdfstring{$\rho$}{rho}-regularity}\label{ss:rho-reg}

Let $K\subset\bR^n$ be some compact (eventually empty) set and let
$\rho:\bR^n\setminus K\to \bR_{\geq 0}$ be a proper analytic submersion.
Let $f:X\to\bR^p$ be a $C^1$-mapping, where $X\subset \bR^n$ is a submanifold.

\begin{definition}[\textbf{$\rho$-regularity at infinity}] \label{d:rho-reg}
We say that $f$ is \textit{$\rho$-regular at $p_0\in\bX^{\infty}$} if there is an open neighborhood $U\subset\bP^n\times\bR^p$ of $p_0$ and an open neighborhood $D\subset\bR^p$ of $\tau(p_0)$ such that, for all $t\in D$, the fibre $f^{-1}(t)\cap U$ intersects all the levels of the restriction $\rho_{|U\cap\bR^n}$ and this intersection is transversal.

We call \emph{Milnor set} the critical locus $\Sing(f, \rho)$ and we denote it by $\cM(f)$.

We say that \textit{the fibre $f^{-1}(t_0)$ is $\rho$-regular at infinity} if $f$ is $\rho$-regular at all
points $p_0\in \bX^{\infty}\cap\tau^{-1}(t_0)$. 
We call:
\[
 S(f):=\{t_0\in\bR^p\mid\exists \{\mathrm{x}_j\}_{j\in \bN}\subset \cM(f), \lim_{j\to\infty}\|\mathrm{x}_j\|=\infty\mbox{ and }\lim_{j\to\infty}f(\mathrm{x}_j)=t_0\}.
\]
the \emph{set of asymptotic $\rho$-nonregular values}. We denote $A_{\rho} := f(\Sing f) \cup S(f)$ and call it the \textit{$\rho$-bifurcation set}.
\end{definition}

\begin{remark}\label{n:rho}
 The definition of $\rho$-regularity  at infinity of a fibre
$f^{-1}(t_0)$ {\em does not depend on any
proper extension of $f$}, since it is equivalent to the following: for any
sequence
$\{ \mathrm{x}_k\}_{k\in \bN} \subset \bR^n$, $\|\mathrm{x}_k\| \to \infty$, $f(\mathrm{x}_k) \to t_0$, there
exists some $k_0= k_0(\{\mathrm{x}_k\}_{k\in \bN})$ such that, if $k\ge k_0$ then $f$ is
transversal to
$\rho$ at
$\mathrm{x}_k$. It also follows from the definition that if $f^{-1}(t_0)$ is $\rho$-regular at
infinity then this fibre has at most isolated singularities.

The transversality of the fibres of $f$ to the levels of $\rho$
is a ``Milnor type'' condition. In case $\rho$ is the Euclidean norm, denoted in this
paper by  $\rho_E$, this condition has been used by John Milnor in the local study of
singular functions 
\cite[\S 4,5]{Mi}.  For complex polynomial functions, transversality to big spheres (i.e.
$\rho_E$-regularity, in our definition) was used in
\cite[pag. 229]{Br} and later in  \cite {NZ}, 
where it is called {\em M-tameness}. The name ``Milnor set'' occurs in  \cite {NZ} too. Distance functions like $\rho$ are also central ingredients in defining regular stratifications, e.g. Mather \cite{Ma}, Kuo \cite{Kuo}, Bekka \cite{Be}.
\end{remark}

\begin{example}
Let $\rho : \bR^n \to \bR_{\ge 0}$, $\rho(x) = (\sum_{i=1}^n | x_i |^{2p_i})^{1/2p}$, where
$(w_1, \ldots , w_n) \in \bN^n$, $p = \lcm \{ 
 w_1, \ldots , w_n\}$ and $w_i p_i = p$, $\forall i$. This function is ``adapted'' to
polynomials which are quasihomogeneous of type $(w_1, \ldots , w_n)$. By using it,
one can show that a value $c\in \bR$ is atypical for such a polynomial if and only if
$c$ is a critical value of $f$ (hence only the value $0$ can be atypical). Namely, let
$E_r :=
\{ x\in \bR^n \mid \rho  (x) < r \}$ for some $r>0$. Then the local Milnor fibre of $f$ at
$0\in \bR^n$ (i.e.
$f^{-1}(c) \cap E_\varepsilon$, for some small enough $\varepsilon$ and $0< | c| \ll
\varepsilon$) is diffeomorphic to the global fibre $f^{-1}(c)$, since $f^{-1}(c)$ is
transversal to
$\partial
\overline{E_r}$, $\forall r\ge \varepsilon$. 
 \end{example}

\subsection{\texorpdfstring{$t$}{t}-regularity implies \texorpdfstring{$\rho_E$}{rho\_E}-regularity}\label{ss:t-reg_rho-reg}

\begin{proposition}\label{p:t-reg_rho-reg}
Let $X\subset \bR^n$ be semi-algebraic and let $f:X\to\bR^p$ be a nontrivial semi-algebraic $C^1$ mapping, where $n>p$. If $f$ is $t$-regular at $p_0\in \bX^\ity$ then $f$ is $\rho_E$-regular at $p_0$.
\end{proposition}
\begin{proof}
We may assume without loss of generality that $p_0 = ([0:0:\ldots:1], 0, \ldots 0)$.
Let $d^{\infty}:  \bX\cap U_n\times\bR^p \to\bR_{\ge 0}$, 
        $(y,t)\mapsto \frac{y_0^2}{y_1^2+\ldots+y_{n-1}^2+1}$ and note that 
$d^{\infty}(y,t)=\frac{1}{\rho_E^2(\mathrm{x})}.$
As usual, we denote by $g$ the projection to the variable $y_0$. At $p_0$, the functions $g^2$ and $d^{\infty}$ differ by a unit, they have the same zero locus $\bX^{\infty}$ and the same levels. Therefore  
$\cC^{\infty}_{p_0} = C_{g,p_0}(\bX\backslash \bX^{\infty} \cap (U_n\times\bR^p)) = C_{g^2,p_0}(\bX\backslash \bX^{\infty} \cap (U_n\times\bR^p))= C_{d^{\infty},p_0}(\bX\backslash \bX^{\infty} \cap (U_i\times\bR^p))$, where the last equality follows by Lemma \ref{l:unit}.  

The $t$-regularity at $p$ (Definition \ref{d:t-reg}) is therefore equivalent to:
\begin{eqnarray}
C_{\tau}(\bP^n\times\bR^p)\cap C_{d^{\infty},p_0}(\bX\backslash \bX^{\infty} \cap U_i\times\bR^p) =\emptyset
\end{eqnarray}
which implies that, in some neighbourhood of $p_0$ intersected with $\bR^n$, the fibres $\{\tau = \mbox{const.}\}$ are
transverse to the levels of the function $d^{\infty}$, which coincide with
the levels of the function $\rho_{E}$.
\end{proof}
\begin{remark}\label{r:r-reg_noimply_t-reg}
The converse of Proposition \ref{p:t-reg_rho-reg} is not true, see Example \ref{ex:paunescuzaharia}.
 
\end{remark}

We have the following general structure result: 
\begin{theorem}\label{t:asympSard}  
 Let $f:X\to\bR^p$ be a $C^2$-mapping on a submanifold $X\subset \bR^n$, where $\dim X >p> 0$. Then:
\begin{enumerate}
\rm \item \textbf{(Fibration theorem)} \it \\
 $S(f)$ is a closed set. 
If $X$ is moreover closed, then $A_{\rho_E} := f(\Sing f)\cup S(f)$ is a closed set and the restriction:
\[ f_| : X \m f^{-1}(A_{\rho_E}) \to \bR^p \m A_{\rho_E}
\]
is a locally trivial fibration over each connected component of $\bR^p \m (A_{\rho_E})$.

 In particular $B(f) \subset A_{\rho_E}$.

\smallskip

\rm \item \textbf{(Semi-algebraic Morse-Sard theorem for asymptotic non-regular values)} \it \\
Assume that $X$ is semi-algebraic and that $f$ is a semi-algebraic mapping.\\
Then $S(f)$ and $A_{\rho_E}$ are semi-algebraic sets of dimension $\le p-1$.

\end{enumerate}

\end{theorem}

 \bigskip

 
By chaining Theorem \ref{t:main} to Proposition \ref{p:t-reg_rho-reg}, we get:
\begin{corollary}\label{c:kos_vs_sf}
Let $f : \bR^n \to \bR^p$ be a nontrivial semi-algebraic $C^2$ mapping. If $f$ is $KOS$-regular at $p_0\in \bX^\ity$ then $f$ is $\rho_E$-regular at $p_0$.
In particular, one has the inclusions:
\begin{equation} \label{eq:strict}
 S(f)\subset K_\ity(f) \mbox{ \ \ and \ \ }  A_{\rho_E} \subset A_{KOS}. \end{equation}\fin
\end{corollary}

The above results prove that Theorem  \ref{t:asympSard} extends the  main result of \cite{KOS}, and this extension turns out to be strict, as showed by the next example. 

\begin{example}\label{ex:paunescuzaharia}\cite{LZ}
 The polynomials $f_{n,q}:\mathbb{C}^{3}\rightarrow\mathbb{C}$,
$f_{n,q}(x,y,z):=x-3x^{2n+1}y^{2q}+2x^{3n+1}y^{3q}+yz$, 
where $n,q\in\mathbb{N}\setminus\{0\}$,
 are $\rho_E$-regular at infinity, more precisely
$S(f_{n.q})=\emptyset$. Then it is shown in \cite{LZ} that $f_{n,q}$ satisfies Malgrange's
condition (hence it is $t$-regular at infinity) for any $t\in\mathbb{C}$ if and only if $n\leq q$.
For $n>q$ we therefore get $\emptyset= S(f_{n.q})\subsetneq K_{\infty}(f_{n.q})\neq \emptyset$. In particular, for $n>q$ the polynomial is $\rho$-regular but not $t$-regular.
\end{example}

\begin{example}\label{ex:fair}
Let $f:\bR^3\to\bR^2$, $f(x,y,z)=(x^2,xy)$.
 We have $\bX^\infty=A\cup B$, where $A=\{([0:0:c:d],(0,t_2))\in\bP^3\times\bR^2\}$ and $B=\{([0:0:0:1],(t_1,t_2))\in\bP^3\times\bR^2\mid t_1>0\}$, and that $\cZ^\infty := \cZ \cap (H^\ity \times \bR^2) = \{([0:0:y:z],(t_1,t_2))\in\bP^3\times\bR^2\}$. Thus $\bX^\infty \subsetneq \cZ^\ity$ and by straightforward computations we get that $f$ is not ``fair'' in the sense of Gaffney at any point $p_0\in\bX^\infty$. Therefore one cannot use Gaffney's approach for $f$. Nevertheless, we still have the equivalence of $t$-regularity with the Malgrange-Gaffney condition (Theorem \ref{t:main}) and  the fibration theorem (Theorem \ref{t:asympSard}), see also Figure \ref{f:spec}.

By straightforward computations one gets $f(\Sing f) = \{ (0,0) \}$ and $B(f) =K(f) = S(f) = \tau(A)=\{(0,t_2);t_2\in\bR\}$. 
\end{example}

\section{Proof of Theorem \ref{t:asympSard}}


\subsection{ Proof of (b).} 
The image $f(\Sing f)$  by $f$ of the semi-algebraic set $\Sing f$ is semi-algebraic,  by the Tarski-Seidenberg theorem, and of dimension $\le p-1$ by the semi-algebraic Sard theorem. 

To show that $S(f)$ is semi-algebraic, we use the semi-algebraic embedding
$\varphi : \bR^n \to \bR^{n+1}\times \bR^p$, 
$(x_{1},\ldots,x_{n}) \mapsto \left(\frac{1}{\sqrt{1+\| \mathrm{x}\|  ^{2}}}, \frac{x_{1}}{\sqrt{1+\| \mathrm{x}\|^{2}}}, \ldots, \frac{x_{n}}{\sqrt{1+\| \mathrm{x}\|^{2}}}, f(\mathrm{x})\right)$.
Let $V_{1}:=\overline{\varphi(\cM(f))}\cap\{(z_0,z_{1},\ldots,z_n,t)\in\mathbb{R}^{n+1}
\times\mathbb{R}^p\mid z_{0}=0\}$ and let $\pi:\mathbb{R}^{n+1}\times\mathbb{R}^p\rightarrow\mathbb{R}^p$ be the canonical projection. Then $V_1$ is semi-algebraic and $S(f)=\pi(V_{1})$, so we may conclude by the Tarski-Seidenberg theorem.

To prove the dimension assertion for $S(f)$ we need some preliminaries.

\subsection*{Partial Thom stratification at infinity}

We show that $\bX^\ity$ may be endowed with a stratification having good enough properties such that one may use it to define the stratified singular locus 
of $\tau_{| \bX^\ity}$. By ``stratifications'' we  mean, as usual, locally finite stratifications satisfying the frontier condition.  For some strata $A, B$, we write $B \prec A$ to say that $B\subset \overline{A} \m A$.

\begin{proposition}\label{p:partialThom}
 Let $f : X \to \bR^p$ be a semi-algebraic $C^1$-mapping on a semi-algebraic subset $X\subset \bR^n$. There exists a semi-algebraic Whitney (a)-regular stratification $\cS$ of $\bX$ such that $\bX^\ity$ is a union of strata, and that any pair of strata $B\prec A$, with $A \subset \bX\m \bX^\ity$ and $B \subset \bX^\ity$,
satisfy the Thom (a$_g$)-regularity condition with respect to some function $g$ defining locally $\bX^\ity$ in $\bX$. 
\end{proposition}
 \begin{proof}
  We follow  \cite[\S 2]{Ti-compo} and start with some  Whitney (a)-regular stratification of $\bX$ with semi-algebraic strata; this exists after Whitney \cite{Wh}, see also \cite[Ch. I]{GWPL}.  One then refines it to a semi-algebraic stratification such that $\bX^\ity$ is a union of strata, see \cite[Ch. I]{GWPL}. Next, since the ($a_g$)-regularity condition is stratifiable (see e.g. \cite[Ch. I]{GWPL}, \cite{DLS}, \cite[\S 3]{Be}), applying the Thom condition to the pairs of strata as in the above statement yields a further refinement which is the desired stratification $\cS$, at least locally.

 We however need to show that this refinement is a \textit{globally} defined stratification of $\bX^\ity$. The argument goes as follows: in the  ($a_g$)-regularity test one uses  the limits at some point of $\bX^\ity$ of the tangent hyperplanes along strata coming from $\bX\m \bX^\ity$. These  limits are precisely described by the space of characteristic covectors at infinity $\cC^\ity$. But by  Lemma \ref{l:unit}, $\cC^\ity$ is independent of the function $g$ defining $\bX^\ity$  locally. 
 \end{proof}

\begin{definition}\label{d:partialThom}
 One calls \textit{partial Thom stratification at infinity} a stratification $\cS$ as in Proposition \ref{p:partialThom}. 
\end{definition}
\noindent
Such stratifications\footnote{have been introduced in \cite[Def. 2.1]{Ti-compo}, \cite[Appendix 1]{Ti-b}, for $p= 1$.} depend of course on the embedding $X\subset \bR^n$. 

\smallskip

 With these notations and definitions, let
$\mathcal{S}=\{S_i\}_{i\in I} $ be a semi-algebraic partial Thom stratification at infinity, the existence of which has been proved above. 
Consider the projection $\tau:\bP^n\times\bR^p\to\bR^p$ and $t_0\in\bR^p$.
The
{\em critical locus at infinity} of the restriction $\tau_{|\bX}$ with respect to $\cS$ is defined as follows:
\[ \Sing^\ity_\cS \tau_{|\bX} := \bigcup_{\cS_i\subset \bX^\ity} \Sing \tau_{| \cS_i}.\]
Since the stratification of $\bX^\ity$ is in particular Whitney (a)-regular, it follows that $\Sing^\ity_\cS \tau_{|\bX}$ is a closed semi-algebraic subset of $\bX^\ity$. Then, by the semi-algebraic Sard theorem, the image $\tau(\Sing^\ity_\cS \tau_{|\bX})\subset \bR^p$ is semi-algebraic and of dimension $\le p-1$.

It also follows from the definition that:
\[
 p_0 \not\in \Sing^\ity_\cS \tau_{|\bX} \Longrightarrow f \mbox{ is $t$-regular at $p_0$, } 
\]
which implies that $\bR^p \m \tau(\Sing^\ity_\cS \tau_{|\bX})$ is included in the set of $t$-regular values of $f$.  Therefore, by Proposition \ref{p:t-reg_rho-reg},  we get the inclusion
 $\bR^p \m \tau(\Sing^\ity_\cS \tau_{|\bX}) \subset \bR^p \m S(f)$,
which shows that $\dim S(f) \le p-1$.


\subsection{Proof of (a).}\label{ss:proofSard}
 Let  $t_0\in \overline{S(f)}$ and let $\{t_{i}\}_{i\in \bN}\subset S(f)$ be a sequence such that $\lim_{i\rightarrow\infty}t_{i}=t_0$.
By definition, for every $i$ we have  a sequence $\{\mathrm{x}_{i,k}\}_{k\in \bN}\subset \cM(f)$
such that $\lim_{k\to \infty} \| \mathrm{x}_{i,k} \| =\infty$ and $\lim_{k\to \infty} f(\mathrm{x}_{i,k})=t_{i}$. For each $i$, there exists
$k(i)\in \bN$ such that if $k\geqslant k(i)$ then $\|\mathrm{x}_{i,k}\| >i$ and $\left| f(\mathrm{x}_{i,k})-t_{i}\right| < 1/i$. Setting $\mathrm{x}_{i}:=\mathrm{x}_{i, k(i)}$, one gets a sequence $\{\mathrm{x}_{i}\}_{i\in \bN}\subset \cM(f)$ such that
 $\lim_{i\to \infty}\|\mathrm{x}_{i}\|=\infty$
and $\lim_{i\to\infty}f(\mathrm{x}_{i})=t_0$. This shows that $t_0\in S(f)$, hence $S(f)$ is closed.

Let us assume now that $X$ is closed and let $t_0\in \overline{f(\Sing f)}\cup S(f)$. We may assume that $t_0\in \overline{f(\Sing f)}$ since we have just proved that $S(f)$ is closed.
Then there exists a sequence $\{\mathrm{x}_j\}_{j\in \bN}\subset\Sing f$,such that $\lim_{j\rightarrow\infty}f(\mathrm{x}_{j})=t_0$.
If $\{\mathrm{x}_j\}_{j\in \bN}$ is non-bounded, we may choose a subsequence
$\{\mathrm{x}_{j_k}\}_{k\in \bN}$ such that $\lim_{k\to\infty} \| \mathrm{x}_{j_k} \| =\infty$
and $\lim_{k\rightarrow\infty}f(\mathrm{x}_{j_k})=t_0$.
Since $\Sing f\subset \cM(f)$, it follows that $t_0\in S(f)$ which is closed. If $\{\mathrm{x}_j\}_{j\in \bN}$ is bounded, then we may choose a
subsequence $\{\mathrm{x}_{j_k}\}_{k\in \bN}$ such that $\lim_{k\to\infty}\mathrm{x}_{j_k}=\mathrm{x}_{0}\in X$ since $X$ is assumed to be closed,
and that $\lim_{k\to \infty}f(\mathrm{x}_{j_k})=t_0$.
Since $\Sing f$ is a closed set, this implies $\mathrm{x}_{0}\in\Sing f$, and we get $t_0=f(\mathrm{x}_{0})\in f(\Sing f)$, which shows that $t_0\in f(\Sing f)\cup S(f)$.

\smallskip
Let us finally show the fibration statement.
We first prove a fibration result in the neighbourhood of  infinity.

\begin{definition}\label{d:triv_at_inf}({\bf Topological triviality at infinity})\\
We say that $f$ is {\em topologically trivial at infinity at the value $t_0\in \bR^p$} if there exists 
a compact set $K\subset \bR^n$ and a ball $B_\delta\subset \bR^p$ centered at $t_0$ such that the restriction:
\begin{equation}\label{eq:triv_at_inf}
f_| : (X\setminus K) \cap f^{-1}(B_\delta) \to  B_\delta
\end{equation}
 is a trivial topological fibration. 
\end{definition}
 Note that one may have two situations in which the mapping (\ref{eq:triv_at_inf}) may be a trivial fibration, namely  whenever $B_\delta\subset\im f$ or when $B_\delta \subset \bR^p \m \im f$.  In the later, the fibration has empty fibre.

\begin{proposition}\label{p:rho-reg_top-triv} \textbf{($\rho$-regularity implies topological triviality at infinity)}\\
 Let $f:X\to\bR^p$ be a $C^2$ mapping, for $n>p$. If the fibre $f^{-1}(t_0)$ is $\rho$-regular at infinity, then $f$ is topologically trivial at infinity at $t_0$.

In particular, $f$ is topologically trivial at infinity at any value of $\bR^p \m S(f)$.
\end{proposition}
\begin{proof}
Let $t_0\not\in S(f)$. Since $S(f)$ is a closed set (as proved just above),
there exists a closed ball $D$ centered at $t_0$ and included in $\bR^p\m S(f)$. Then there exists some large enough radius $R_0 \gg 0$ such that:\begin{equation}\label{eq:milnor-set-at-inf}
 \cM(f)\cap f^{-1}(D)\m  B_{R_0}^n = \emptyset. 
\end{equation}
 Indeed, if this were not true, then there exists a sequence $\{\mathrm{x}_{k}\}_{k\in\mathbb{N}}\subset f^{-1}(D)\cap \cM(f)$
with $\lim_{k\to \ity}\|\mathrm{x}_{k}\|=\infty$ and since $D$ is compact, one may extract a sub-sequence $\{\mathrm{x}_{k_i}\}_{i\in\mathbb{N}}\subset \cM(f)$ with  $\lim_{k\to \ity}f(\mathrm{x}_{k_i})=t \in D$, 
which gives a contradiction to $D\cap S(f) = \emptyset$.

To prove the topological triviality at infinity at $t_0$ it is enough to show that the mapping:
\begin{equation}\label{eq:triv2}
 f_| : f^{-1}(D)\m B_{R}^{n} \to D
\end{equation}
is a trivial fibration on the manifold with boundary $(f^{-1}(D)\m B_{R}^{n}, f^{-1}(D)\cap S_{R}^{n-1})$, for any $R \ge R_0$. This is a submersion by (\ref{eq:milnor-set-at-inf}) but it is not proper, so one cannot apply Ehresmann's theorem directly. Instead, we consider the map $(f,\rho) : f^{-1}(D)\m B_{R}^{n} \to D\times [R, \infty[$. Now, as a direct consequence of its definition, this is a proper map. It is still a submersion by (\ref{eq:milnor-set-at-inf}) and since $\Sing (f,\rho) = \cM(f)$. We then apply Ehresmann's theorem to the mapping $(f,\rho)$ in order to conclude that it is a locally trivial fibration, hence trivial  over $D\times [R, \infty[$. Take now the projection $\pi : D\times [R, \infty[ \to D$ which is a trivial fibration by definition and remark that our map (\ref{eq:triv2}) is the composition $\pi \circ (f,\rho)$ of two trivial fibrations, hence a trivial fibration too.
\end{proof}

\begin{remark}
 It is interesting to point out that the implication in the above proposition is \textit{not an equivalence} in general. The paper \cite{TZ} presents an example of a polynomial mapping $f: \bR^2 \to \bR$ which is not $\rho_E$-regular at infinity at the value $0$, thus not $KOS$-regular either, but it is $C^\ity$ trivial at infinity at $0$.  It is easier to give such examples in the topological cathegory, for instance $f(x,y) = x^3$ which is topologically equivalent to the projection on $x$, whereas $S(f) = \{0\}$.
\end{remark}

We now complete the proof of Theorem \ref{t:asympSard}(a).
Since $\bR^p \m (f(\Sing f)\cup S(f))$ is an open set, for any fixed $t_0\not\in f(\Sing f)\cup S(f)$ there exists a closed ball $D$ centered at $t_0$ such that $D\subset \bR^p\m f(\Sing f)\cup S(f)$. By the above proof of Proposition \ref{p:rho-reg_top-triv} and using the same notations, one has the trivial fibration (\ref{eq:triv2}) on the manifold with boundary $(f^{-1}(D)\m B_{R}^{n}, f^{-1}(D)\cap S_{R}^{n-1})$, for any $R \ge R_0$.

Next, since $D\cap f(\Sing f) = \emptyset$, the restriction:
\begin{equation}\label{eq:triv1}
 f_| : f^{-1}(D)\cap \bar B_{R_0}^{n} \to D
\end{equation}
is a proper submersion on the manifold with boundary $(f^{-1}(D)\cap \bar B_{R_0}^{n}, f^{-1}(D)\cap S_{R_0}^{n-1})$ and therefore a locally trivial fibration by Ehresmann's theorem, hence a trivial fibration over $D$.

We finally glue together the two trivial fibrations \eqref{eq:triv1} and \eqref{eq:triv2} by using an isotopy and the trivial fibration from the following commutative diagram, for some $R > R_0$:
\begin{equation}\label{eq:triv3}
\begin{array}{c}\xymatrix{ 
 (\bar B_{R} \m \mathring{B}_{R_0}) \cap f^{-1}(D) \ar@{>}[d]^{\simeq}   \ar@{>}[r]^{\ \ \ \ \ \ (f,\rho)} & D \times [ R_0, R] \ar@{>}[d]^{\pr} \\
\hat F \times  D \times [ R_0, R] \ar@{>}[ur]  \ar@{>}[r]^{\ \ \ \ \ \ \ \pr}  &  D }
\end{array}
\end{equation}
where $\hat F$ denotes the fibre of the trivial fibration $f_| : S_R \cap f^{-1}(D) \to D$ and does not depend on the radius $R > R_0$.
\fin

\begin{remark}\label{r:X-closed}
 If in Theorem \ref{t:asympSard}(a) we do not assume that $X$ is closed then the fibration assertion holds if one replaces in the statement $f(\Sing f)$ by its closure $\overline{f(\Sing f)}$.
\end{remark}





\end{document}